\documentclass{amsart}
\usepackage{mathrsfs}
\usepackage{amsfonts}
\usepackage{txfonts}
\usepackage{amsmath}
\usepackage{amssymb}
\usepackage{amsthm}
\usepackage{graphicx}
\usepackage[toc,page,title,titletoc,header]{appendix}
\usepackage{geometry}
\usepackage{upgreek}
\usepackage[T1]{fontenc}
\usepackage{color}
\usepackage{hyperref}
\usepackage[all]{xy}
\usepackage[french,english]{babel}

\theoremstyle{plain}
\newtheorem{theo}{Theorem}[section]
\newtheorem*{theo*}{Theorem}
\newtheorem{prop}{Proposition}[section]

\newtheorem*{lem*}{Lemma}

\theoremstyle{definition}

\newtheorem{claim}{Claim}[section]
\newtheorem{question}{Question}[section]

\theoremstyle{remark}

\newtheorem*{ack}{Acknowledgements}
\newtheorem*{rem*}{Remark}

\newcommand{\R}{\mathbb{R}}

\begin{document}

\title{A Note on the Passing Time of the Spherical Kepler Problem}

\author{Lei Zhao}
\address{Chern Institute of Mathematics, Nankai University/Institute of Mathematics, University of Augsburg}
\email{l.zhao@nankai.edu.cn}

\abstract 

In this note we collect some known facts concerning central projection correspondances and time parametrizations of Kepler problems in Euclidean spaces and on Spheres. 

\endabstract

\date\today
\maketitle

\section{Introduction}
{With the help of central projection}, the Kepler problem on the sphere with a cotangent potential as defined by P. Serret \cite{Serret} can be effectively deduced from the classical Kepler problem in Euclidean space (\cite[Lecture 5]{Albouy}). The classical and spherical Kepler problems share many similarities. To name a few:

\begin{itemize}
\item The orbits are planar resp. spherical conic sections, with one focus at the attracting center (P. Serret \cite{Serret});
\item The energy depends only on the (geodesic) semi major axis (Killing \cite{Killing}, Velpry \cite{Velpry});
\item The assertion of Bertrand theorem (c.f. \cite[Lecture 3]{Albouy}) holds (Liebmann \cite{Liebmann}, Higgs \cite{Higgs}, Ikeda and Katayama\cite{IkedaKatayama}, Kozlov and Harin \cite{Kozlov}).
\end{itemize}

All these phenomena can be derived from the orbital correspondence of the corresponding systems in the {Euclidean space} and on the sphere via central projections. 

Concerning time parametrization of Kepler problem in Euclidean space, Lambert's theorem states the following: For simplicity we consider the planar problem. {For two points} $A_{1}$ and $A_{2}$ in a plane, the passing time from $A_{1}$ to $A_{2}$ along an arc of the Keplerian orbit with attracting center $O$ and semi major axis $a$ is a multivalued function of the energy of the orbit and the three mutual distances $c=|A_{1} A_{2}|, r_{1}=|A_{1} O|, r_{2}=|O A_{2}|$. Indeed, {when the energy is negative,} the energy determines the semi major axis length, which measures the size of the elliptic orbit, while $r_{1}, r_{2}, c$ subsequently determine (in a multi-value way) the eccentricity of the orbit. Lambert's theorem \cite{Lambert} asserts that this dependence on four variables can be effectively reduced to three: the energy, or equivalently the semi major axis $a$, the mutual distance $c$ and the sum $r_{1}+r_{2}$. {The two functions $c$ and $r_{1}+r_{2}$ are well-defined for all triples of positive real numbers $(r_{1}, r_{2}, c)$ and are independent of the semi major axis $a$. We note that {this fact} is also true for parabolic and hyperbolic orbits, corresponding respectively to cases with zero and positive Kepler energies.}

In the spherical problem, the passing time is analogously a multivalued function of the spherical energy and the three geodesic distances on the sphere. Again, the spherical energy determines the geodesic semi major axis. An analogue of Lambert's theorem on the sphere would thus be a similar reduction of number of variables on which  the passing time depends. 

\begin{question}\label{Question: 1} In the spherical Kepler problem, do there exist two functions $f, g$ of the three geodesic distances among the start point, the end point and the attracting center, so that the passing time from the {start point to the end point along an Keplerian arc of the orbit} can be expressed as a multivalued function only of these two functions $f, g$ and the geodesic semi major axis (or equivalently the energy) of the orbit?
\end{question}

In Section \ref{Section: Spherical Kepler Problem}, we shall recall the definition and some properties of the spherical Kepler problem from \cite[Lecture 5]{Albouy}. We then analyse the passing time along Keplerian arcs and Lambert's Theorem in Section \ref{Section: Passing Time Kepler Planar}.  In Section \ref{Section: Passing Time}, we analyze the passing time in the spherical problem. 

\textbf{Note: In early versions, the author thought have got an negative answer to Question \ref{Question: 1} with a computer-assisted proof.
However mistakes has been found with the help of a surprising proof communicated by Alain Albouy, which again uses the central projection correspondences and in particular shows that the answer to this question should be positive. }
The initial purpose of this note thus falls out.

\section{Central Projection and the Spherical Kepler Problem} \label{Section: Spherical Kepler Problem}
Let $(\mathbf{F}, \langle, \rangle)$ be a three-dimensional Euclidean space and 
$$\mathbf{E}:=\{q \in \mathbf{F}: <Z,q>=R\}$$
 be a plane in $\mathbf{F}$ for some {$Z \in \mathbf{F}$ of unit norm, and some $R>0$.}
For a particle $q$ moving in $\mathbf{F}$, let $h= \langle Z,q \rangle/R$ be the ``normalized height'' of $q$. For those $q \in F$ such that $\langle Z, q \rangle \neq 0$, we consider the motion of its central projection $q_\mathbf{E}=q/h$ in $\mathbf{E}$. 
We have 
$$ \dot{q}_\mathbf{E}=\dfrac{h \dot{q}- \dot{h } q}{h^2},\, \dfrac{d}{dt}(h^2 \dot{q}_\mathbf{E})=h \ddot{q}-\ddot{h} q.$$  

\begin{figure}
\center
\includegraphics[width=70mm]{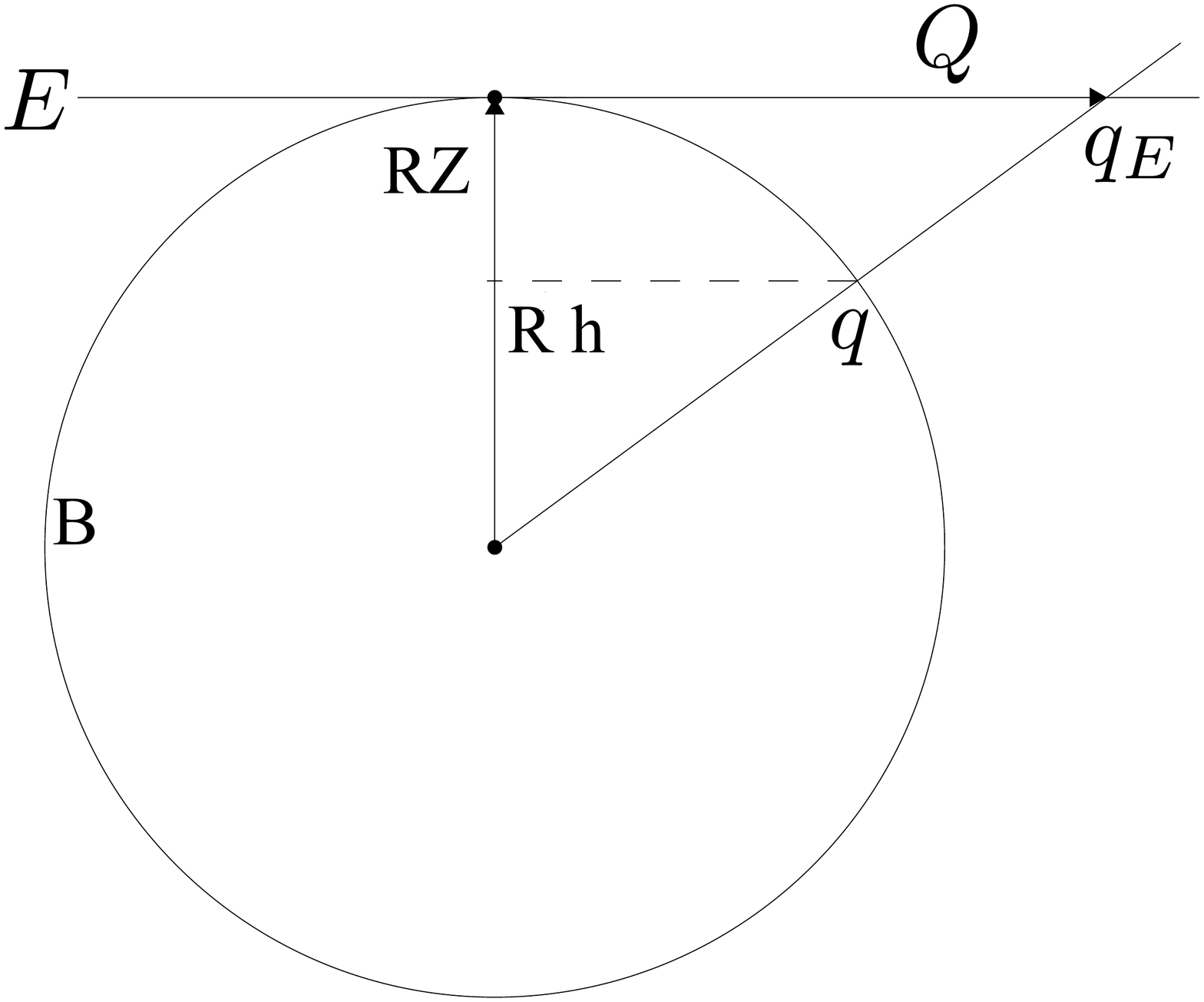}
\caption{Central projection of a sphere (from \cite{Albouy})}
\end{figure}

Let the motion of $q$ be given by $\ddot{q}=\psi(q)+\lambda q$, where $\lambda$ is a function of $q,\dot{q}$ and also possibly a function of time $t$. We have then 
\begin{equation}\label{eq: tau t}
\dfrac{d}{dt}(h^2 \dot{q}_\mathbf{E})=h \psi(q) - \dfrac{<Z,\psi(q)>}{R} q.
\end{equation}

Denote by $B:=B(R)$ the sphere in $\mathbf{F}$ centered at the origin with radius $R$. {From now on, the motion of $q$ is restricted on $B$.} We have

\begin{prop}(Appell \cite{Appell}, see also \cite{Albouy}) \label{Correspondence} A force field on $B$ induces a force field on $\mathbf{E} \subset \mathbf{F}$, s.t. the central projection of the moving particle $q \in B$ moves under the induced force field on $\mathbf{E}$ up to a time reparametrization $d \tau=h^{-2} d t$.
\end{prop}

A central field on $B$, with an attracting center at the contact point $O=R Z$ of $\mathbf{E}$ and $B$, thus naturally induces a central field on $\mathbf{E}$. When the induced central field on $\mathbf{E}$ is that of the Kepler problem, the corresponding force field on $B$ defines the spherical Kepler problem. 

The unit normal vector $Z$ of $\mathbf{E}$ {can be extended} to a constant vector field {in $\mathbf{F}$ normal to $\mathbf{E}$ } for which we still denote it by $Z$. By projecting $Z$ orthogonally to the tangent planes of $B$ we get a central force field $Z_B=Z-\dfrac{h}{R} q$ on $B$, in which $h=\dfrac{\langle Z,q \rangle}{R}$. The centers of the central force field are respectively $R Z$  and $-R Z$. Any other central force field with the same centers can be expressed as $b(q) \, Z_B$, where $b: B \setminus \{R Z,-R Z\} \rightarrow \R$ is a scalar function.

In order to calculate the central force field on $\mathbf{E}$ corresponding to $b \, Z_B$, let $Q=q_\mathbf{E}-R Z$, and recall that $\tau$ is a new time variable defined by $ d\tau=h^{-2} d t$. From Eq.\eqref{eq: tau t} with $\psi(q)=b Z_{B}$, we have 
$$\dfrac{d^2 Q}{d \tau^2}=\dfrac{d^2 q_\mathbf{E}}{d \tau^2}=h^{2} \dfrac{d}{dt} (h^{2} \dot{q}_{\mathbf{E}})=bh^2(h Z_B-\dfrac{\langle Z,Z_B \rangle}{R} q)=bh^2 (h Z- \dfrac{q}{R})=-\dfrac{b h^3 Q}{R}$$ 
and 
\begin{equation}\label{eq: doth}
h=\dfrac{R}{\|q_\mathbf{E}\|} =\dfrac{R}{\sqrt{R^2+\|Q\|^2}}=\dfrac{1}{\sqrt{1+\frac{\|Q\|^2}{R^2}}},
\end{equation}
which is equivalent to
$$\|Q\|=R h^{-1} (1-h^{2})^{\frac12}.$$

In order to have 
$$\dfrac{d^2 Q}{d \tau^2}=-\|Q\|^{-3} Q,$$
i.e. the induced central force field defines the Kepler problem\footnote{We suppose throughout this note that the masses are unit.} on $\mathbf{E}$, we need to have 
$$\dfrac{b h^3}{R}=\|Q\|^{-3}, \hbox{ that is } b=R^{-2} (1-h^2)^{-\frac{3}{2}}.$$

By Proposition \ref{Correspondence}, the Keplerian orbits on $B$ are spherical conics (c.f.  \cite{SykesPeirce}) and the attracting center is one of their foci. Its energy takes the form 
$$\mathcal{E}=\dfrac{1}{2} \|\dot{q}\|^{2}+U(q),$$
in which the function $U(q)$ (that we suppose moreover to be homogeneous to eliminate the constant of integration) satisfies 
$$\nabla U(q)=\psi(q)=b \, Z_{B}.$$ 
By a direct calculation, we find
$$U(q)=-\dfrac{\langle q, Z \rangle }{R \sqrt{\|q\|^{2}-\langle q, Z \rangle^{2} }}=-\dfrac{h}{R \sqrt{1-h^{2}}}.$$

We now write the energy $\mathcal{E}$ with the semi major axis $a$ and the eccentricity $e$ of the projected Keplerian orbit in $\mathbf{E}$:

\begin{claim} 
\begin{equation}\label{Eq: Energies}
\mathcal{E} =-\dfrac{1}{2 a} + \dfrac{a (1-e^{2})}{2 R^{2}}.
\end{equation}
\end{claim}
\begin{proof}In terms of $Q$, the expression of $U$ is 
$$U=-\dfrac{h}{R\sqrt{1-h^{2}}}=-\dfrac{1}{\|Q\|}.$$

Now we calculate $ \|\dot{q}\|^{2}$. From $q=h \, q_{\mathbf{E}}$ we get $\dot{q}=\dot{h} \, q_{\mathbf{E}}+ h \, \dot{q}_{\mathbf{E}}$. Note also $q_{\mathbf{E}}=R Z+Q$. We have thus $\dot{q}_{\mathbf{E}}=\dot{Q}$ and $\langle\dot{q}_{\mathbf{E}},q_{\mathbf{E}}\rangle=\langle\dot{Q},Q\rangle$. From \eqref{eq: doth}, we also have 
$$\dot{h}=-\dfrac{\langle \dot{Q},Q \rangle}{R^{2} (\sqrt{1+\|Q\|^{2}/R^{2}})^{3}}=-\dfrac{ h^{3} \langle \dot{Q},Q \rangle}{R^{2}}.$$
 Therefore
\begin{align*}
 \|\dot{q}\|^{2}&=\langle \dot{h} \, q_{\mathbf{E}}+h \, \dot{q}_{\mathbf{E}},\dot{h} \, q_{\mathbf{E}}+h \, \dot{q}_{\mathbf{E}}\rangle \\&=\dot{h}^{2} \, \|q_{\mathbf{E}}\|^{2}+h^{2}\, \|\dot{q}_{\mathbf{E}}\|^{2}+ 2 h \, \dot{h}\, \langle q_{\mathbf{E}},\dot{q}_{\mathbf{E}}\rangle \\&=\dot{h}^{2} \, \dfrac{R^{2}}{h^{2}}+h^{2} \, \|\dot{Q}\|^{2}+ 2 h \, \dot{h} \langle Q,\dot{Q}\rangle \\&=h^{2} \, \|\dot{Q}\|^{2}-\dfrac{h^{4}}{R^{2}} \, \langle Q, \dot{Q} \rangle^{2} \\&=h^{-2} \, \|\dfrac{d Q}{d \tau}\|^{2}-\dfrac{1}{R^{2}} \, \langle Q, \dfrac{d Q}{d \tau} \rangle^{2} \\&= \|\dfrac{d Q}{d \tau}\|^{2} + \dfrac{1}{R^{2}} \, (\|Q\|^{2} \|\dfrac{d Q}{d \tau}\|^{2}-\langle Q, \dfrac{d Q}{d \tau} \rangle^{2})\\&= \|\dfrac{d Q}{d \tau}\|^{2} + \dfrac{C^{2}}{R^{2}}.
\end{align*}
In which $C:= \|Q \times \dfrac{d Q}{d \tau}\|$ is the norm of the angular momentum of the Keplerian orbit in $\mathbf{E}$. 
As by definition
$$\mathcal{E}=\dfrac{1}{2} \|\dot{q}\|^{2}-\dfrac{1}{\|Q\|},$$
we have
\begin{equation*}
\mathcal{E}=E+\dfrac{C^{2}}{2 R^{2}},
\end{equation*}
where we have denoted by $E$ the energy of the projected Kepler problem in $\mathbf{E}$. We {thus get Eq. \eqref{Eq: Energies} from} the classical expressions $E=-\dfrac{1}{2 a}$, $C^{2}=a(1-e^{2})$.
\end{proof}

{On the other hand, if we denote the geodesic major axis by $R \theta_{a}$, where $\theta_{a}$ is the maximal central angle of the spherical ellipse. It is already known to Killing  (c.f. \cite{Killing}, \cite{Velpry}) that $\theta_{a}$ is only a function of $\mathcal{E}$. Let us calculate a little more to see that this is indeed the case: By central projection, the pericenter and apocenter of the spherical ellipse in $B$ is projected respectively to the pericenter and apocenter of the projected ellipses in $\mathbf{E}$. The line passing through the origin of $\mathbf{F}$ and the point $R Z$ is perpendicular to $\mathbf{E}$, and divides the angle $\theta_{a}$ into two angles $\theta'$ and $\theta''$. The major axis of the projected ellipse in $\mathbf{E}$ is divided by the focus point $O=R Z$ into two segments with length $a(1+e)$ and $a(1-e)$ respectively. As the vector $R Z$ has length $R$, we have 
$$\{\tan \theta', \tan \theta''\}=\{\dfrac{a(1-e)}{R}, \dfrac{a(1+e)}{R}\}.$$
Therefore 
$$\tan \theta_{a}=\tan (\theta' +\theta'')=\dfrac{\frac{2 a}{R}}{1-\frac{a^{2} (1-e^{2})}{R^{2}}}=\dfrac{1}{R(\frac{1}{2 a}-\frac{a (1-e^{2})}{2 R^{2}})}=-\dfrac{1}{R\mathcal{E}}.$$
And thus for fixed $R$, the angle $\theta_{a}$ is only a function of $\mathcal{E}$.}

\section{Lambert Theorem of the Kepler Problem} \label{Section: Passing Time Kepler Planar}
For two points $A_{1}$ and $A_{2}$ in $\mathbf{E}$, the passing time $\Delta T$ from $A_{1}$ to $A_{2}$ along a Keplerian arc of an elliptic orbit with semi major axis $a$ can be expressed as a function of the three mutual distances $r_1=|OA_1|$, $r_2=|OA_2|$, $c=|A_1 A_2|$ and $a$. In general, there are two ellipses passing through the points $A_{1}, A_{2}$ with semi major axis $a$. We denote by $u$ the eccentric anomaly of a point on an elliptic orbit, and denote by $u_{1}, u_{2}$ the eccentric anomalies of $A_{1}, A_{2}$ respectively. There are yet many choices for these angles and thus, it is seen that $\Delta T$ is a multivalued function of $r_{1}, r_{2}, c$ and $a$.

Nevertheless, once $(u_{1}, u_{2}, e, a)$ are chosen, we deduce directly from the Kepler equation that
\begin{claim} 
\begin{equation}\label{calculation of Kepler passing time}
\Delta T = \int_{u_1}^{u_2} a^{\frac{3}{2}}\, (1-e \cos u)\,du .
\end{equation}
\end{claim}

Following Lagrange \cite{Lagrange}, we define two angles $\phi,\psi$ by the relations
\begin{equation}\label{Coordinate Change 1}
\left\{\begin{array}{l}
\phi=\arccos \bigl(e \cos((u_{1}+u_{2})/2)\bigr); \\
\psi=(u_{1} - u_{2})/2.
\end{array}
\right.
\end{equation}
The change of coordinates from $(u_{1}, u_{2}, e)$ to $(\phi, \psi, e)$ is regular when 
$$e \neq 0, e \neq 1, u_{1}+u_{2} \neq 0\, (mod\, 2 \pi). $$
{In general this change of coordinates is given by a two-to-one mapping,} with $(u_{1}, u_{2}, e)$ and $(u'_{1}=-u_{2}, u'_{2}=-u_{1}, e)$ corresponding to the same value of $(\phi, \psi, e)$.  

By substitution, we see that $\Delta T$, as a function of $(\phi, \psi, e, a)$, does not depend on $e$:
$$\Delta T =a^{\frac{3}{2}} (u_{2}-u_{1}-e (\sin u_{2}-\sin u_{1}))=a^{\frac{3}{2}}(-2 \psi + 2 \sin \psi \cos \phi).$$ 
{On the other hand, the following relations have been deduced by Lagrange in \cite[pp. 564-566]{Lagrange}}
\begin{equation} \label{relation edges angles}
\left\{\begin{array}{l}
r_1+r_2=2a \,(1-\cos \psi \cos \phi) \\
c=2a \,  |\sin \psi \sin \phi|,
\end{array}
\right.
\end{equation}
{which allows to express $\psi$ and $\phi$ as multi-valued functions only of $r_1+r_2$, $a$ and $c$.} We may thus conclude with Lagrange that 
\begin{theo} (Lambert \cite{Lambert}) $\Delta T$ is a multi-valued function of $c$, $r_{1}+r_{2}$ and $a$.
\end{theo}

Many proofs of Lambert's theorem were found in the history. In \cite{AlbouyLambert}, a timeline of these proofs can be found, together with two new approaches to this theorem.

\section{Period of periodic orbits in the spherical Kepler problem} \label{Section: Passing Time}
We shall now consider the corresponding spherical Kepler problem on $B$ defined above. The point $O$ is the North pole of $B$. Let $B_1, B_2$ be two points on the north hemisphere. Let $s_1, s_2, d$ be respectively the geodesic distances $O B_1$, $O B_2$ and $B_1 B_2$ on $B$, and let $\theta_1=s_1/R$, $\theta_2=s_2/R$, $\theta=d/R$. Let $\Delta' T$ be the passing time for a particle to move from $B_1$ to $B_2$ along a spherical Keplerian orbit in a spherical ellipse with energy $\mathcal{E}$. 

We shall eventually analyze $\Delta' T$ in $\mathbf{E}$ while keeping the time reparametrization in mind. We suppose that one of the ellipses determined by $O$, $B_{1}, B_{2}$ and $\mathcal{E}$ lies entirely in the north hemisphere, so that it projects to an ellipse in $\mathbf{E}$. 

Let $r=\|Q\|$ be the distance of the projected particle $q_{\mathbf{E}}$ to $O$. In terms of the elliptic elements of the projected ellipse in $\mathbf{E}$, we have 
$$r=a (1-e \cos u)$$
 and
$$h=\dfrac{1}{\sqrt{1+\frac{r^{2}}{R^{2}}}}=\dfrac{R}{\sqrt{R^{2}+a^{2} (1-e\cos u)^{2}}}.$$
{Also, by differentiating the Kepler equation in $\mathbf{E}$, we obtain}
$$d \tau =a^{\frac32} (1-e \cos u) d u.$$
{Thus with $d t = h^{2}d \tau=\dfrac{d \tau}{1+\frac{r^{2}}{R^{2}}}$, we obtain}
\begin{equation}\label{calculation of passing time}
\Delta' T = \int_{t_{1}}^{t_{2}} d t =\int_{\tau_{1}}^{\tau_{2}} \dfrac{1}{1+\frac{r^{2}}{R^{2}}} d \tau = \int_{u_1}^{u_2} a^{\frac{3}{2}} \dfrac{1-e \cos u}{1+\frac{a^2}{R^2}(1-e \cos u)^2} \,du.
\end{equation}

By normalization we set $R=1$ from now on.

The period of the spherical elliptic motion has been calculated in \cite{Killing}, \cite{Kozlov}, and is found to depend only on $\mathcal{E}$. Indeed, we may calculate the integral 
$$\mathbf{T}=\int_{0}^{2 \pi} a^{\frac{3}{2}} \dfrac{1-e \cos u}{1+a^2(1-e \cos u)^2} \,du$$
{by decomposing the integrand as}
$$a^{\frac{3}{2}} \dfrac{1-e \cos u}{1+a^2(1-e \cos u)^2}=\dfrac{\sqrt{a}}{2} (\dfrac{1}{i +a(1-e \cos u)}+\dfrac{1}{-i +a(1-e \cos u)})$$
which allows us to integrate them out separately. We thus get
$$\mathbf{T}=\pi \sqrt{a} (\dfrac{1}{\sqrt{a^{2}(1-e^{2})-1+2 a i}}+ \dfrac{1}{\sqrt{a^{2}(1-e^{2})-1-2 a i}})$$
{in which both complex square roots are meant to have positive real parts.}

With $R=1$,  Eq. \eqref{Eq: Energies} now reads
\begin{equation}\label{eq: energy with unit radius}
\mathcal{E}=-\dfrac{1}{2 a} +\dfrac{a (1-e^{2})}{2}
\end{equation}
thus
\begin{equation}\label{eq: relation for integration}
a^{2} (1-e^{2})-1=2 a \mathcal{E}.
\end{equation}

{From this,  we have}
$$\mathbf{T}=\dfrac{\pi \sqrt{2}}{2} (\dfrac{1}{\sqrt{ \mathcal{E}+ i}}+ \dfrac{1}{\sqrt{ \mathcal{E}- i}}).$$
{By squaring this expression and taking square root, we obtain in the end}
$$\mathbf{T}=\dfrac{\pi \sqrt{\mathcal{E} + \sqrt{\mathcal{E}^{2}+1}}}{\sqrt{\mathcal{E}^{2}+1}}.$$

\begin{ack} Thanks to Alain Albouy, Otto van Koert and the referee for many discussions and helps.
\end{ack}

\end{document}